\documentclass[a4paper,10pt]{article}
\usepackage{amsmath}
\usepackage{amsfonts}
\usepackage{amssymb}
\usepackage{amscd}
\usepackage{pb-diagram}
\usepackage{epic}
\usepackage{url}
\usepackage{color}
\usepackage{amsthm}
\title{Remarks on the space of volume preserving embeddings}
\author{Mathieu Molitor\\
\it \small{Department of Mathematics, Keio University}\\
\it \small{3-14-1, Hiyoshi, Kohoku-ku, 223-8522, Yokohama, Japan}\\ 
\small{\it{e-mail:}}\,\,\url{pergame.mathieu@gmail.com}
}
\date{}
\begin{document}

\newtheorem{lemme}{Lemma}[section]
\newtheorem{definition}[lemme]{Definition}
\newtheorem{proposition}[lemme]{Proposition}
\newtheorem{corollaire}[lemme]{Corollary}
\newtheorem{theoreme}[lemme]{Theorem}
\newtheorem{remarque}[lemme]{Remark}
\newtheorem{example}[lemme]{Example}
\bibliographystyle{alpha}

\maketitle

\begin{abstract}
	Let $(N,g)$ be a Riemannian manifold. For a compact, connected and oriented submanifold $M$ of $N\,,$ 
	we define the space of volume preserving embeddings $\text{Emb}_{\mu}(M,N)$ as the set of smooth
	embeddings $f\,:\,M\hookrightarrow N$ such that $f^{*}\mu^{f}=\mu\,,$ where $\mu^{f}$ (resp. $\mu$)
	is the Riemannian volume form on $f(M)$ (resp. $M$) induced by the ambient metric $g$ 
	(the orientation on $f(M)$ being induced by $f$). \\
	In this article, we use the Nash-Moser inverse function Theorem to show that the set of volume preserving embeddings in 
	$\textup{Emb}_{\mu}(M,N)$ whose 
	mean curvature is nowhere vanishing forms a tame Fr\'echet manifold,
	and determine explicitly the Euler-Lagrange equations of a natural class of Lagrangians.
	
	As an application, we generalize the Euler equations of an incompressible fluid to the case of an ``incompressible membrane"
	of arbitrary dimension moving in $N\,.$
\end{abstract}
\section*{Introduction}
	Fluid mechanics and infinite dimensional geometry already share a long and common history. In 1966, Arnold \cite{Arnold} 
	suggested to regard the space of velocity fields of an incompressible fluid as the Lie algebra of the infinite dimensional 
	Lie group of volume preserving diffeomorphisms :
	\begin{eqnarray}
		\textup{SDiff}_{\mu}(M):=\big\{\phi\in \textup{Diff}(M)\,\vert\,\phi^{*}\mu=\mu\big\}\,.
	\end{eqnarray}
	Here $M$ is the oriented manifold on which the fluid is living, $\mu$ is the volume form of $M$ and 
	$\textup{Diff}(M)$ is the group of all smooth diffeomorphisms of $M\,.$ In this setting, Arnold interpreted 
	the Euler equations of an incompressible fluid as a geodesic equation on $\textup{SDiff}_{\mu}(M)$ 
	for an appropriate right-invariant metric. 
	
	It was not until the 70's that Arnold's vision of fluid mechanics could be made partially rigorous with the development of 
	Banach and Hilbert manifolds. In \cite{Ebin-Marsden}, Ebin and Marsden considered volume preserving diffeomorphisms on a compact 
	manifold $M$ 
	which are not smooth, but of Sobolev classes. In doing so, they obtained topological groups locally modelled on Hilbert spaces, 
	and were able, following Arnold's ideas, to prove analytical results on the Euler equations. Their method is still 
	an active research area (see for example \cite{Gay-well,Gay-Ratiu}). 
	
	On the geometrical side, volume preserving diffeomorphisms which are not smooth are problematic. For, the 
	left-multiplication $L_{\phi}\,:\,\textup{Diff}(M)\rightarrow \textup{Diff}(M)\,,\,\,\psi\mapsto \phi\circ\psi$ 
	consumes derivatives of $\phi\,,$ and thus, subgroups of the group of diffeomorphisms whose elements are \textit{not} 
	smooth cannot be turned into genuine infinite dimensional Lie groups (left multiplication is not smooth). Hence, 
	from a Lie group theory point of view, one has to consider the group of \textit{smooth} volume preserving 
	diffeomorphisms of $(M,\mu)\,,$ i.e., the group $\textup{SDiff}_{\mu}(M)\,.$

	For technical reasons, $\textup{SDiff}_{\mu}(M)$ can only be given a Lie group structure modelled on topological vector spaces 
	which are more general than Banach and Hilbert spaces, and an inverse function
	theorem, applicable beyond the usual Banach space category, is necessary. To our knowledge, only two authors succeeded in doing this.
	The first was Omori who showed and used an inverse function theorem in terms of 
	ILB-spaces (``inverse limit of Banach spaces", see \cite{Omori}), and later on, Hamilton with 
	his category of tame Fr\'echet spaces together with the Nash-Moser inverse function 
	Theorem (see \cite{Hamilton}). Nowadays, it is nevertheless not uncommon 
	to find  mistakes or big gaps in the literature when it comes to the differentiable 
	structure of $\textup{SDiff}_{\mu}(M)\,,$ even in some specialized textbooks in infinite 
	dimensional geometry. The case of $M$ being non-compact is even worse, and no proof that $\textup{SDiff}_{\mu}(M)$ 
	is a ``Lie group" is available in this case. 

	A natural generalization of $\textup{SDiff}_{\mu}(M)\,,$ with which we shall be concerned in this paper,
	is the space of \textit{volume preserving embeddings} 
	$\textup{Emb}_{\mu}(M,N)\,.$ This space is defined as follows. For a Riemannian manifold $(N,g)$ and a compact, connected and 
	oriented submanifold $M$ of $N\,,$
	\begin{eqnarray}	
		\textup{Emb}_{\mu}(M,N):=\Big\{f\in\textup{Emb}(M,N)\,\Big\vert\,f^{*}\mu^{f}=\mu\Big\}\,,
	\end{eqnarray} 
	where $\textup{Emb}(M,N)$ is the space of smooth embeddings from $M$ into $N\,,$ 
	and where $\mu^{f}$ (resp. $\mu$) is the Riemannian volume form on $f(M)$ (resp. $M$) induced by the 
	ambient metric $g$ (the orientation on $f(M)$ being induced by $f$). 
	
	When $M$ is an open subset of $\mathbb{R}^{n}$ with boundary\footnote{In this paper, all manifolds have no boundary.}, 
	then it is possible to extend Arnold's method by introducing a $L^{2}$-metric on $\textup{Emb}_{\mu}(M,N)$ 
	and to show that the corresponding geodesics describe 
	the dynamics of a liquid drop with free boundary. This has been discussed formally in \cite{Lewis}, 
	and rigorous results in this direction 
	can be obtained using spaces of volume preserving embeddings of Sobolev classes, 
	as pointed out to us by Sergiy Vasylkevych\footnote{Private communication.}.

	In this paper, we focus on \textit{smooth} volume preserving embeddings, i.e., on the space $\textup{Emb}_{\mu}(M,N)$ as defined 
	above. To this end, we adopt a rigorous infinite dimensional point of view based on
	Hamilton's category of tame Fr\'{e}chet manifolds, and determine explicitly a natural class of Lagrangian equations 
	on $\textup{Emb}_{\mu}(M,N)$. We allow $M$ to be of arbitrary dimension, and we assume that it has no boundary.

	More precisely, using the techniques developed by Hamilton in \cite{Hamilton}, as well as a generalization of the Helmholtz-Hodge 
	decomposition Theorem for vector fields supported 
	on submanifolds (Proposition \ref{decomposition de Hel-Hod}), we are able, in Theorem \ref{Emn variete tame}, to show the following result: 
	\textit{the space $\textup{Emb}_{\mu}(M,N)^{\times}$ of volume preserving embeddings 
	whose mean curvature is nowhere vanishing forms a tame Fr\'{e}chet submanifold of $\textup{Emb}(M,N)\,.$} This result is a consequence 
	of the Nash-Moser inverse function Theorem. 
	
	Having a manifold structure on $\textup{Emb}_{\mu}(M,N)^{\times}\,,$ we then consider Lagrangian mechanics on it. The Lagrangians we consider 
	are of the following form:
	\begin{eqnarray}
		 \widetilde{L}(X_{f}):=\int_{M}\,L\circ X_{f}\cdot \mu\,,
	\end{eqnarray}
	where $L\,:\,TN\rightarrow \mathbb{R}$ is a Lagrangian density and where 
	$X_{f}\,:\,M\rightarrow TN$ is a ``divergence free vector field along $f$", regarded as an element of 
	$T_{f}\textup{Emb}_{\mu}(M,N)^{\times}\,.$ As it turns out, the resulting 
	Euler-Lagrange equations are (pointwise) the usual finite dimensional Euler-Lagrange equations (written in a covariant form),
	twisted by a ``Helmholtz-Hodge projection" (Proposition \ref{theoreme ultime??? lol}). 

	When $L$ is the energy associated to the metric $g\,,$ then the corresponding Euler-Lagrange equations on $\textup{Emb}_{\mu}(M,N)^{\times}$ 
	are geodesic equations which generalize the Euler equations of an incompressible fluid to the case of an ``incompressible 
	membrane" of arbitrary dimension moving in $N$ (Proposition \ref{je derpime pas de fric}). 

	It would be interesting to know if these equations have a physical meaning.\\

	The paper is organized as follows. In \S\ref{subsection the category of hamitlon}, we review very briefly 
	Hamilton's category of tame Fr\'{e}chet manifolds. In \S\ref{sss madco encore...}, we show that 
	$\textup{Emb}_{\mu}(M,N)^{\times}$ is a tame Fr\'{e}chet submanifold of $\textup{Emb}(M,N)\,;$ this requires a generalization 
	of the Helmholtz-Hodge decomposition. In \S\ref{sss pas inspireeee}, we compute the Euler Lagrange equations on 
	$\textup{Emb}_{\mu}(M,N)^{\times}$ for a natural class of Lagrangians, and in \S\ref{eeeeffffgggfggrrrr}, we identify 
	the natural generalization of the Euler equations of an incompressible fluid.

\section{The differentiable structure of the space of volume preserving embeddings}	
	Let $(N,g)$ be a Riemannian manifold and let $M$ be a compact, connected and 
	oriented submanifold of $N\,.$ We denote by $\text{Emb}(M,N)$ 
	the space of smooth embeddings from $M$ into $N\,.$

	For an embedding $f\,:\,M\hookrightarrow N\,,$ we denote by $\mu^{f}$ the volume form on $f(M)$ induced by the 
	restriction of the metric $g$ to the submanifold $f(M)$ 
	(the orientation on $f(M)$ being induced by $f$). With this terminology, 
	we define the space of \textit{volume preserving embeddings} as
	\begin{eqnarray}
		 \textup{Emb}_{\mu}(M,N):=\Big\{f\in\textup{Emb}(M,N)\,\Big\vert\,f^{*}\mu^{f}=\mu\Big\}\,,
	\end{eqnarray}
	where $\mu$ is the Riemannian volume form on $M$ induced by the metric $g\,.$\\

	The aim of this section is to use the Nash-Moser inverse function Theorem (as formulated in \cite{Hamilton}) 
	to define a differentiable structure on the open subset
	\begin{eqnarray}
		\text{Emb}_{\mu}(M,N)^{\times}:=
			\Big\{f\in\text{Emb}_{\mu}(M,N)\,\Big\vert\,(\textup{Tr}\,\Pi_{f})_{x}\neq 0\textup{ for all }x\in f(M)\Big\}\,,
	\end{eqnarray}
	where $\textup{Tr}\,\Pi_{f}$ denotes the trace of the second fundamental form of $f(M)\,.$	

	For the reader's convenience, let us recall that the second fundamental form $\Pi_{f}$ of 
	the submanifold $f(M)$ is defined, for $x\in f(M)$ and for two vector fields $X,Y$ on $f(M)\,,$ by 
	\begin{eqnarray}
		(\Pi_{f})_{x}(X,Y):=\nabla_{\widetilde{X}}\widetilde{Y}-\nabla^{f}_{X}Y\,,
	\end{eqnarray}
	where $\nabla$ (resp. $\nabla^{f}$) is the Levi-Civita connection on $N$ (resp. $f(M)$) 
	induced by $g$ (resp. $g\vert_{f(M)}$), and where $\widetilde{X},\widetilde{Y}$ are vector fields on $N$ extending $X$ and $Y\,.$

	Let us also recall that the trace of the second fundamental form $\Pi_{f}$ is defined, for $x\in f(M)\,,$ by 
	\begin{eqnarray}
		(\textup{Tr}\,\Pi_{f})_{x}:=\sum_{i=1}^{k}\,\Pi_{f}(e_{i},e_{i})\,,
	\end{eqnarray}	
	where $k$ is the dimension of $M$ and where $\{e_{1},...,e_{k}\}$ is an orthonormal basis for $T_{x}f(M)\,.$ 
	In particular, $\textup{Tr}\,\Pi_{f}$ is a section of the normal bundle $\textup{Nor}_{f}$ of $f(M)\,,$ the latter bundle 
	being, by definition, the vector bundle over $f(M)$ whose fiber over $x\in f(M)$ 
	is
	\begin{eqnarray}
		(\textup{Nor}_{f})_{x}:=\big\{u_{x}\in T_{x}N\,\big\vert\,
		g_{x}(u_{x},v_{x})=0\,\,\textup{for all}\,\,v_{x}\in T_{x}f(M)\big\}\,.
	\end{eqnarray}	
	Finally, recall that $\textup{Tr}\,\Pi_{f}$ is, up to a multiplicative constant which depend on convention, 
	the mean curvature of the submanifold $f(M)\,.$

\subsection{Hamilton's category of tame Fr\'echet manifolds}
\label{subsection the category of hamitlon}
	In this section, we review very briefly the category of tame Fr\'{e}chet manifolds 
	introduced by Hamilton in \cite{Hamilton}. 
\begin{definition}
		\begin{description}
			\item[$(i)$] A graded Fr\'echet space $(F,\{\|\,.\,\|_{n}\}_{n\in\mathbb{N}})\,,$
				is a Fr\'echet space $F$ whose topology is defined by a collection of seminorms 
				$\{\|\,.\,\|_{n}\}_{n\in\mathbb{N}}$ which are increasing in strength:
				\begin{eqnarray}
					\|x\|_{0}\leq\|x\|_{1}\leq\|x\|_{2}\leq \cdots
				\end{eqnarray}
				for all $x\in F\,.$
			\item[$(ii)$] A linear map $L\,:\,F\rightarrow G$ between two graded Fr\'echet spaces
				$F$ and $G$ is tame (of degree $r$ and base $b$) if for all $n\geq b\,,$
				there exists a constant $C_{n}>0$ such that for all $x\in F\,,$
				\begin{eqnarray}
					\|L(x)\|_{n}\leq C_{n}\,\|x\|_{n+r}\,.
				\end{eqnarray}
			\item[$(iii)$] If $(B,\|\,.\,\|_{B})$ is a Banach space, then $\Sigma(B)$ denotes 
				the graded Fr\'echet space of all sequences $\{x_{k}\}_{k\in\mathbb{N}}$ of 
				$B$ such that for all $n\geq 0,$ 
				\begin{eqnarray}
					\|\{x_{k}\}_{k\in\mathbb{N}}\|_{n}:=\displaystyle\Sigma_{k=0}^{\infty}\,
					e^{nk}\|x_{k}\|_{B}<\infty\,.
				\end{eqnarray}
			\item[$(iv)$] A graded Fr\'echet space $F$ is tame if there exist a Banach space 
				$B$ and two tame linear maps $i\,:\,F\rightarrow \Sigma(B)$ and 
				$p\,:\,\Sigma(B)\rightarrow F$ such that $p\circ i$ is the identity on $F\,.$
			\item[$(v)$] Let $F,G$ be two tame Fr\'echet spaces, $U$ an open subset of 
				$F$ and $f\,:\,U\rightarrow G$ a map. We say that $f$ is a smooth tame map
				if $f$ is smooth\footnote{By smooth we mean that 
				$f\,:\,U\subseteq F\rightarrow G$ 
				is continuous and that 
				for all $k\in\mathbb{N}\,,$ the $k$th derivative 
				$d^{k}f\,:\,U\times F\times \cdots \times F
				\rightarrow G$ exists and is jointly continuous on the product  space, such as 
				described in \cite{Hamilton}.\label{footnote}} 
				and if for every $k\in\mathbb{N}$ and for every 
				$(x,u_{1},...,u_{k})\in U\times F\times \cdots F\,,$ there exist a neighborhood
				$V$ of $(x,u_{1},...,u_{k})$ in $U\times F\times \cdots F$ and 
				$b_{k},r_{0},...,r_{k}\in\mathbb{N}$ such that for every $n\geq b_{k}\,,$ there 
				exists $C_{k,n}^{V}>0$ such that 
				\begin{eqnarray}
					&&\|d^{k}f(y)\{v_{1},...,v_{k}\}\|_{n}
					\,\,\leq \,\,C_{k,n}^{V}\,\big(1+\|y\|_{n+r_{0}}
					+\|v_{1}\|_{n+r_{1}}+\cdots+\|v_{k}\|_{n+r_{k}}\big)\,,
				\end{eqnarray}
				for every $(y,v_{1},...,v_{k})\in V\,,$ where $d^{k}f\,:\,
				U\times F\times\cdots\times F\rightarrow G$ denotes the $k$th derivative of 
				$f\,.$
		\end{description}
\end{definition}
\begin{remarque}
	In this paper, we use interchangeably the notation 
	$(df)(x)\{v\}$ or $f_{*_{x}}v$ for the first derivative of $f$ at a 
	point $x$ in direction $v\,.$ 
\end{remarque}
	As one may notice, tame Fr\'echet spaces and smooth tame maps form 
	a category, and it is thus natural to define a tame Fr\'echet manifold 
	as a Hausdorff topological space with an atlas of coordinates charts taking their value in
	tame Fr\'echet spaces, such that the coordinate transition functions are all
	smooth tame maps (see \cite{Hamilton}).
	The definition of a tame smooth map between tame Fr\'echet manifolds is then 
	straightforward, and we thus obtain a subcategory of the category of Fr\'echet manifolds.\\
	In order to avoid confusion, let us also make precise our notion of submanifold. We will say 
	that a subset $\mathcal{M}$ of a 
	tame Fr\'echet manifold $\mathcal{M}\,,$ endowed with the trace topology, 
	is a submanifold, 
	if for every point $x\in \mathcal{M}\,,$ there exists a chart 
	$(\mathcal{U},\varphi)$ of $\mathcal{M}$ such that $x\in \mathcal{U}$ and such that 
	$\varphi(\mathcal{U}\cap \mathcal{M})=U\times \{0\}\,,$ where 
	$\varphi(\mathcal{U})=U\times V$ is a product of two open subsets of tame Fr\'echet spaces. 
	Note that a submanifold of a tame Fr\'echet manifold is also a tame Fr\'echet manifold.

\begin{theoreme}[Nash-Moser inverse function Theorem, \cite{Hamilton}]
	Let $F,G$ be two tame Fr\'echet spaces, $U$ an open subset of 
	$F$ and $f\,:\, U\rightarrow G$ a smooth tame map. If there exists 
	an open subset $V\subseteq U$ such that 
	\begin{description}
		\item[$(i)$] $df(x)\,:\,F\rightarrow G$ is an linear isomorphism for all 
			$x\in V\,,$
		\item[$(ii)$] the map $V\times G\rightarrow F,\,(x,v)\mapsto 
			\big(df(x)\big)^{-1}\{v\}$ is a smooth tame map, 
	\end{description}
	then $f$ is locally invertible on $V$ and each local inverse is a
	smooth tame map. 
\end{theoreme}

\subsection{The differentiable structure of $\text{Emb}_{\mu}(M,N)^{\times}$}\label{sss madco encore...}
	Let $\Sigma$ be an oriented submanifold of $N$ endowed with the Riemannian volume form $\mu^{\Sigma}$ 
	induced by $g\,.$ 

	We shall use the following terminology:
	\begin{description}
		\item[$\bullet$] $TN\vert_{\Sigma}$ is the restriction of the bundle $TN$ to 
			$\Sigma$ with associated space of sections $\Gamma(TN\vert_{\Sigma})\,.$ 
		\item[$\bullet$] For a  vector field $X\in \mathfrak{X}(\Sigma)\,,$ 
				$\textup{div}_{\Sigma}(X)$
				is the divergence of $X$ with respect to the volume form $\mu^{\Sigma}\,,$ 
				i.e., it is the only function which satisfies 
				$\mathcal{L}_{X}\mu^{\Sigma}=\textup{div}_{\Sigma}(X)\cdot \mu^{\Sigma}\,,$
				where $\mathcal{L}_{X}$ is the Lie derivative in direction $X\,.$
		\item[$\bullet$] $\Gamma_{\mu}(TN\vert_{\Sigma}):=\big\{X\in \Gamma(TN\vert_{\Sigma})\,\big\vert\, 
				\textup{div}_{\Sigma}(X^{\top})-g(X^{\perp},\textup{Tr}\,\Pi_{\Sigma})=0\big\}\,,$
				where $X^{\top}$ and $X^{\perp}$ are respectively the tangential and orthogonal 
				projections of $X$ on the tangent and normal bundles of $\Sigma\,.$ 
	\end{description}
	If $\Sigma=f(M)$ for some embedding $f\in \textup{Emb}(M,N)\,,$ then we shall 
	replace $``\Sigma"$ by $``f"$ in the above notation. For example, $\textup{div}_{f}$ instead of $\textup{div}_{\Sigma}\,,$ etc. 
\begin{proposition}\label{decomposition de Hel-Hod}
	Let $\Sigma$ be a compact, connected, oriented submanifold of $N$ whose mean curvature is not identically zero. 
	Then, for every section $X$ of $TN\vert_{\Sigma}\,,$ 
	there exist a unique $X_{\mu}\in \Gamma_{\mu}(TN\vert_{\Sigma})$ and a unique function $p\,:\,\Sigma\rightarrow \mathbb{R}$ such that
\begin{eqnarray}\label{decomposition explicite}
		 X=X_{\mu}+\textup{grad}(p)+p\cdot\textup{Tr}\,\Pi_{\Sigma}\,,
	\end{eqnarray}
	where $\textup{grad}(p)\in \mathfrak{X}(\Sigma)$ is the Riemannian gradient of $p$ taken with respect to $g\vert_{\Sigma}\,.$
\end{proposition}
\begin{proof}
	Let $X$ be an element of $\Gamma\big(TN\vert_{\Sigma}\big)\,.$ 
	If $X$ could be written $X=X_{\mu}+\textup{grad}(p)+p\cdot\textup{Tr}\,\Pi_{\Sigma}$ 
	with $X_{\mu}\in\Gamma_{\mu}\big(TN\vert_{\Sigma}\big)$ 
	and $p\in C^{\infty}(\Sigma,\mathbb{R})\,,$ then $p$ would be a solution of the 
	following partial differential equation :
	\begin{eqnarray}\label{eq HH}
		 \textup{div}_{\Sigma}(X^{\top})-g(X^{\perp},\textup{Tr}\,\Pi_{\Sigma})=
		\triangle p-\Vert\textup{Tr}\,\Pi_{\Sigma}\Vert^{2}\cdot p\,.
	\end{eqnarray}
	The differential operator $\triangle p-\Vert\textup{Tr}\,\Pi_{\Sigma}\Vert^{2}\cdot p$ acting 
	on functions $p\,:\,\Sigma\rightarrow\mathbb{R}$ is an operator of the 
	form $\triangle-c\,,$ where $c$ is a smooth function, and, being an 
	elliptic operator, it is well known that this operator is Fredholm, and that its analytical index is a 
	topological invariant (see \cite{Palais-seminar}). Hence, the 
	index of $\triangle-c$ equals the index of $\triangle\,,$ which is zero on the space 
	$C^{\infty}(\Sigma,\mathbb{R})\,.$ Moreover, as $c$ is nonnegative, and since 
	$\Sigma$ connected, we can use the maximum principle (see for example \cite[p.96]{Aubin} or 
	\cite[Thm 24.10, p.355]{Jost-postmodern3}), to deduce that the kernel of 
	$\triangle-c\,:\,C^{\infty}(\Sigma,\mathbb{R})\rightarrow C^{\infty}(\Sigma,\mathbb{R})$ 
	is included in the space of constant functions, and as $c$ is not identically zero in our case, 
	this kernel has to be trivial. Hence, $\triangle-c$ is bijective, and 
	\eqref{eq HH} posses a unique solution $p\in C^{\infty}(\Sigma,\mathbb{R})\,.$

	Now, if we take a function $p$ solution to \eqref{eq HH} and set $X_{\mu}:
	=X-\textup{grad}(p)-p\cdot\textup{Tr}\,\Pi_{\Sigma}\,,$ then, it is straightforward to check 
	that $X=X_{\mu}+\textup{grad}(p)+p\cdot\textup{Tr}\,\Pi_{\Sigma}$ is the desired decomposition.
\end{proof}
	Proposition \ref{decomposition de Hel-Hod} yields a topological decomposition 
	\begin{eqnarray}\label{decomposition H-H}
	 	\Gamma\big(TN\vert_{\Sigma}\big)=
		\Gamma_{\mu}\big(TN\vert_{\Sigma}\big)\oplus\Gamma_{\mu}\big(TN\vert_{\Sigma}\big)^{\bot}\,,
	\end{eqnarray}
	where 
	\begin{eqnarray}
		\Gamma_{\mu}\big(TN\vert_{\Sigma}\big)^{\bot}:=\big\{\textup{grad}(p)+p\cdot\textup{Tr}\,\Pi_{\Sigma}\in 
		\Gamma(\textup{Nor}_{\Sigma})\,\big\vert\, 
		p\in C^{\infty}(\Sigma,\mathbb{R})\,\big\}\,.
	\end{eqnarray}
	Moreover, by application of Stokes' Theorem, one easily sees that 
	\eqref{decomposition H-H} is an orthogonal decomposition (whence the notation 
	$``\Gamma_{\mu}\big(TN\vert_{\Sigma}\big)^{\bot}"$) with respect to the following weak scalar product :
	\begin{eqnarray}
		 \Gamma\big(TN\vert_{\Sigma}\big)\times
		\Gamma\big(TN\vert_{\Sigma}\big)
		\rightarrow\mathbb{R}\,,\,\,\,\,(X,Y)\mapsto \int_{\Sigma}\,g(X,Y)\cdot\mu^{\Sigma}\,.
	\end{eqnarray}
\begin{remarque}$\text{}$ \label{remarque sur la decomposition, la vraie de H-H}
	\begin{center}
	\begin{description}
		\item[$(i)$] Proposition \ref{decomposition de Hel-Hod} also holds for 
			$\Sigma=N\,.$ In this case, \eqref{decomposition H-H} reduces to the well known Helmholtz-Hodge 
			decomposition for vector fields (see for example \cite[p.341]{Arnold}) :
			\begin{eqnarray}\label{equation de la decomposition de HH, la vraie, encore}
				 \mathfrak{X}(N)=\mathfrak{X}_{\mu}(N)\oplus \textup{grad}\big(C^{\infty}(N,\mathbb{R})\big)\,,
			\end{eqnarray}
			where $\mathfrak{X}_{\mu}(N):=\{X\in\mathfrak{X}(N)\,\vert\,\textup{div}(X)=0\}$ 
			is the space of divergence free vector fields on $N$ (here $\mu$ denotes the Riemannian volume form of $(N,g)$). 
		\item[$(ii)$] A direct consequence of the existence of a topological direct 
			summand for $\Gamma_{\mu}\big(TN\vert_{\Sigma}\big)$ in 
			the tame Fr\'echet space 
			$\Gamma\big(TN\vert_{\Sigma}\big),$ is that the space 
			$\Gamma_{\mu}\big(TN\vert_{\Sigma}\big)$ is also a tame 
			Fr\'echet space (see \cite{Hamilton}).
		\item[$(iii)$] We shall denote by $\mathbb{P}_{\Sigma}\,:\,\Gamma\big(TN\vert_{\Sigma}\big)
			\rightarrow\Gamma_{\mu}\big(TN\vert_{\Sigma}\big)$ 
			the continuous projection given by Proposition \ref{decomposition de Hel-Hod}.
	\end{description}\end{center}
\end{remarque}
	Let $\text{Emb}(M,N)^{\times}$ be the open subset of $\text{Emb}(M,N)$
	defined by 
	\begin{eqnarray}
		\textup{Emb}(M,N)^{\times}:=
		\Big\{f\in\textup{Emb}(M,N)\,\Big\vert\,(\textup{Tr}\,\Pi_{f})_{x}\neq 0\textup{ for all }x\in f(M)\Big\}\,.
	\end{eqnarray}
	The global version of Proposition \ref{decomposition de Hel-Hod} is :
\begin{theoreme}\label{Emn variete tame}
	The space $\textup{Emb}_{\mu}(M,N)^{\times}$ is a tame 
	Fr\'echet submanifold of the Fr\'echet manifold 
	$\textup{Emb}(M,N)^{\times},$ and for $f\in \textup{Emb}_{\mu}(M,N)^{\times},$ 
	we have the following natural isomorphism 
	\begin{eqnarray}
		T_{f}\textup{Emb}_{\mu}(M,N)^{\times}\cong\Gamma_{\mu}\big(f^{*}TN\big),
	\end{eqnarray}
	where $\Gamma_{\mu}\big(f^{*}TN\big):=
	\big\{X\in \Gamma(f^{*}TN)\,\big\vert\, X\circ f^{-1}\in \Gamma_{\mu}\big(TN\vert_{f(M)}\big)\big\}\,.$
\end{theoreme}
	In order to show Theorem \ref{Emn variete tame}, let us recall the construction of the ``standard'' 
	chart $(U_{f},\varphi_{f})$ of $\textup{Emb}(M,N)$ centered at a point $f\in\textup{Emb}(M,N)\,.$ For this, 
	we need
	\begin{description}
		\item[$\bullet$] a sufficiently small neighborhood 
			$\Theta_{f}\subseteq f^{*}TN$ of the zero section of $f^{*}TN$ such that 
			the map $\Theta_{f}\rightarrow N\times N,\,v_{x}\in\Theta_{f}\cap 
			T_{x}N\mapsto\big(x,\textup{exp}_{x}(v_{x})\big)$ 
			(here the exponential map $\textup{exp}$ is taken with 
			respect to the metric $g$) is a diffeomorphism onto an open subset of $N\times N\,.$
		\item[$\bullet$] $\varphi_{f}(U_{f}):=\{X\in\Gamma\big(f^{*}TN\big)\,
			\vert\,X(M)\subseteq\Theta_{f} \}\,.$
		\item[$\bullet$] $\varphi_{f}^{-1}\,:\,\varphi_{f}(U_{f})\rightarrow U_{f}$ 
			is defined for $X\in\varphi_{f}(U_{f})$ and $x\in M\,,$ 
			by $\big(\varphi_{f}^{-1}(X)\big)(x):=\textup{exp}_{f(x)}X_{x}\,.$ 
			For brevity's sake, we shall write 
			$f_{X}:=\varphi_{f}^{-1}(X)\in \text{Emb}(M,N)$ ($f_{X}$ can be seen as a 
			``perturbation of the embedding $f$ by the vector field $X$'').
	\end{description}
	It is well known that $\textup{Emb}(M,N)$ endowed with these charts is a tame Fr\'echet 
	manifold (see for example \cite{Hamilton,Kriegl-Michor}), and it is clear, 
	restricting the open sets $\Theta_{f}$ if necessary, 
	that we also get an atlas for $\text{Emb}(M,N)^{\times}\,.$\\

	We will also need the following map	
	\begin{eqnarray}
		\text{Emb}(M,N)\overset{\rho}{\longrightarrow} 
		C^{\infty}(M,\mathbb{R}),\,\,\,\,f\mapsto f^{*}\mu^{f}/\mu\,,
	\end{eqnarray}
	i.e., for $f\in\text{Emb}(M,N),$  $\rho(f)$ 
	is the unique function satisfying $f^{*}\mu^{f}=\rho(f)\cdot \mu$ on $M\,.$ Observe that $\rho(f)>0\,.$
	
	Finally, for $f\in\text{Emb}(M,N)^{\times}\,,$ we define
	\begin{center}
		\begin{description}
			 \item[$\bullet$] $P_{f}\,:\,\varphi_{f}(U_{f})\rightarrow 
				C^{\infty}(M,\mathbb{R}),\,X\mapsto (\rho\circ \varphi_{f}^{-1})(X)$ ($P_{f}$ 
				is nothing but the local expression of $\rho$ in the chart $(U_{f},\varphi_{f})).$ 
			\item[$\bullet$] 
				$Q_{f}\,:\,\varphi_{f}(U_{f})\rightarrow 
				\Gamma_{\mu}\big(TN\vert_{f(M)}\big)\oplus 
				C^{\infty}(M,\mathbb{R}),\,X\mapsto\big(X_{\mu},P_{f}(X)-1\big),$
				where we use Proposition \ref{decomposition de Hel-Hod} to write 
				$X=(X_{\mu}+\textup{grad}(p)+p\cdot\textup{Tr}\,\Pi_{f(M)})\circ f\,.$
	\end{description}\end{center}
	Observe that $f_{X}=\varphi_{f}^{-1}(X)\in U_{f}$ is volume preserving if and only if $P_{f}(X)\equiv1\,.$\\

	Following Hamilton in \cite[Thm.2.5.3]{Hamilton}, if we prove that 
	$Q_{f}$ is a local diffeomorphism near the zero section, then it would be 
	possible, using $Q_{f}\circ\varphi_{f}\,,$ to define splitting charts for 
	$\text{Emb}(M,N)^{\times}\,,$ and thus to prove that $\text{Emb}_{\mu}(M,N)^{\times}$ 
	is a tame submanifold of $\text{Emb}(M,N)^{\times}\,.$ We will do this with two lemmas, 
	the main point being the use of the inverse function theorem of Nash-Moser.

\begin{lemme}\label{expression locale de rho}
	The map $P_{f}\,:\,\varphi_{f}(U_{f})\rightarrow C^{\infty}(M,\mathbb{R})$ is a 
	smooth tame map, and its derivative $(P_{f})_{*_{X}}Y$ is a family of linear 
	partial differential operators of degree 1 in $Y$ with coefficients which 
	are nonlinear partial differential operators of degree 1 in $X.$ Moreover, 
	\begin{eqnarray}\label{expression de la derivée de Pf....}
		 (P_{f})_{*_{0}}Y=\Big[\textup{div}_{f}(Y^{\top}\circ f^{-1})-g(Y^{\perp}\circ f^{-1},\textup{Tr}\,\Pi_{f})\Big]
		\circ f\cdot P_{f}(0),
	\end{eqnarray}
	where $Y\in \Gamma\big(f^{*}TN\big)\,.$
\end{lemme}
\begin{proof}
	Let $\big(U,\phi=(x_{1},...,x_{m})\big)$ be a positively oriented chart for $M$ and 
	$\big(V,\psi=(y_{1},...y_{n})\big)$ a chart for $N$ such that $f_{X}(U)\subseteq V$ for $X$ sufficiently small.

	For $x\in U,$ a direct calculation shows that
 	\begin{eqnarray}\label{expression locale de rho en coordonnées locales}
		\big(P_{f}(X)\big)(x)=\dfrac{\textup{det}\big(g_{f_{X}(x)}((f_{X})_{*_{x}}
		\partial_{x_{i}},(f_{X})_{*_{x}}\partial_{x_{j}})\big)^{1/2}}{\textup{det}
		\big(g_{x}(\partial_{x_{i}},\partial_{x_{j}})\big)^{1/2}}.
	 \end{eqnarray}
	As $(f_{X})_{*_{x}}\partial_{x_{i}}=\textup{exp}_{*}X_{*_{x}}\partial_{x_{i}},$ 
	we see that \eqref{expression locale de rho en coordonnées locales} is a 
	nonlinear differential operator of degree 1 in $X\in \Gamma\big(f^{*}TN\big)$, 
	and it is well known that a nonlinear differential operator is a smooth tame map (see \cite[Cor. 2.2.7]{Hamilton}).

	Let us now compute its derivative in local coordinates. For this purpose, let us take a 
	smooth curve of sections $X_{t}$ in $\varphi_{f}(U_{f}).$ We shall denote 
	$Y_{t}:=\partial_{t}X_{t},$ and $f_{t}:=f_{X_{t}}=\varphi_{f}^{-1}(X_{t})$ the 
	corresponding smooth curve of embeddings in $\textup{Emb}(M,N)$ 
	(in the following, we may forget the subscript ``t''). After elementary differential calculus, one finds, 
	\begin{eqnarray}
		(P_{f})_{*_{X}}Y&=&\partial_{t}\big(P_{f}(X_{t})\big)=
		\dfrac{d}{dt}\bigg\vert_{t}\dfrac{\textup{det}\big(g_{f_{t}(x)}((f_{t})_{*_{x}}\partial_{x_{i}},(f_{t})_{*_{x}}
		\partial_{x_{j}})\big)^{1/2}}{\textup{det}\big(g_{x}(\partial_{x_{i}},
		\partial_{x_{j}})\big)^{1/2}}\nonumber\\
		&=&1/2\cdot P_{f}(X)\cdot\textup{Tr}\,(A^{-1}\partial_{t}A),\label{expression locale avec la trace}
	\end{eqnarray}   
	 where $A$ is the matrix whose entries are $A_{ij}:=g_{ij}^{f_{t}(M)}:=
	g_{f_{t}}\big((f_{t})_{*}\partial_{x_{i}},(f_{t})_{*}\partial_{x_{j}}\big).$
	To carry out the calculation of $\partial_{t}A$ in local coordinates, we will also denote 
	$g^{N}_{ij}:=g(\partial_{y_{i}},\partial_{y_{j}}),$ $Z_{t}:=\partial_{t}f_{t}\in 
	\Gamma(f_{t}^{*}TN)$ and 
	$\Gamma_{\alpha\beta}^{k}\in C^{\infty}(V,\mathbb{R})$ the 
	Christofell symbols associated to the metric $g$ on $V\,.$ Using Einstein 
	summation convention and the formula 
	$\partial_{y_{\alpha}}g_{ab}^{N}=
	\Gamma_{\alpha a}^{k}g_{kb}^{N}+\Gamma_{\alpha b}^{k}g_{ak}^{N},$ it is then easy to see that
	\begin{eqnarray}\label{expression locale penible de la derivées des composantes de la matrice...}
			\partial_{t}A_{ij}&=&\Gamma_{\alpha a}^{k}{\circ} f\cdot g^{N}_{kb}{\circ} 
			f\cdot Z^{\alpha}\cdot \partial_{x_{i}}f^{a}\cdot\partial_{x_{j}}f^{b}\nonumber\\
		&&+\Gamma_{\alpha b}^{k}{\circ} f\cdot g^{N}_{ka}{\circ} 
			f\cdot Z^{\alpha}\cdot \partial_{x_{i}}f^{a}\cdot\partial_{x_{j}}f^{b}\nonumber\\
		&&\text\;\;+g^{N}_{ab}{\circ}f\cdot\partial_{x_{i}}Z^{a}\cdot
			\partial_{x_{j}}f^{b}+g^{N}_{ab}{\circ}f\cdot\partial_{x_{i}}f^{a}\cdot\partial_{x_{j}}Z^{b}.
	\end{eqnarray}
	Since $Z_{t}=\partial_{t}f_{t}=\partial_{t}f_{X_{t}}=\textup{exp}_{*_{X_{t}}}Y_{t}$ and 
	$\partial_{x_{i}}f^{a}=\partial_{x_{i}}f_{X_{t}}^{a}=\textup{exp}_{*}\partial_{x_{i}}X$ 
	can be considered respectively as a partial differential operator of order 0 
	(nonlinear) in $X$ and (linear) in $Y,$ and a nonlinear partial differential operator of order 1 
	in $X,$ it follows easily, in view of \eqref{expression locale avec la trace} and 
	\eqref{expression locale penible de la derivées des composantes de la matrice...}, that 
	$(P_{f})_{*_{X}}Y$ is a family of linear partial differential operators of degree 1 in $Y$ 
	with coefficients which are nonlinear partial differential operators of degree 1 in $X.$

	Formula \eqref{expression de la derivée de Pf....} can be obtained after direct calculations, 
	splitting $Z$ into its tangential and normal parts $Z^{\top},$ $Z^{\bot},$ and using, 
	among others, equation \eqref{expression locale penible de la derivées des composantes de la matrice...}, 
	$\partial_{y_{\alpha}}g_{ab}^{N}=
	\Gamma_{\alpha a}^{k}g_{kb}^{N}+\Gamma_{\alpha b}^{k}g_{ak}^{N}$ as 
	well as $g^{N}_{ab}{\circ}f\cdot(Z^{\bot})^{a}\cdot\partial_{x_{i}}f^{b}=0.$ One finds
	\begin{eqnarray}
		&&(P_{f})_{*_{X}}Y=1/2\cdot P_{f}(X)\cdot\textup{Tr}\,(A^{-1}\partial_{t}A)=P_{f}(X)\cdot\big[\nonumber\\
		&&(g^{f(M)})^{ij}{\circ}f\cdot g^{N}_{ab}{\circ}f\cdot 
			\partial_{x_{i}}f^{a}(\partial_{x_{j}}(Z^{\top})^{b}+(Z^{\top})^{\alpha}\cdot 
			\partial_{x_{j}}f^{\beta}\cdot \Gamma_{\alpha\beta}^{b}{\circ}f)\label{equation de la divergence en local}\\
		&&\text{}\;\;-(g^{f(M)})^{ij}{\circ}f\cdot g^{N}_{ab}{\circ}f\cdot (Z^{\bot})^{a}(\partial^{2}_{x_{i}x_{j}}f^{b}
			-\partial_{x_{j}}f^{\alpha}\cdot\partial_{x_{i}}f^{\beta}\cdot
			\Gamma^{b}_{\alpha\beta}{\circ}f)\big].
			\label{equation avec la trace de la seconde forme fondamentale blablabla}
	\end{eqnarray}
	One recognizes \eqref{equation de la divergence en local} as being 
	$\textup{div}_{f}(Z^{\top}{\circ}f^{-1}){\circ}f$ and 
	\eqref{equation avec la trace de la seconde forme fondamentale blablabla} to be 
	the negative of $g_{f}(Z^{\bot}{\circ}f^{-1},\textup{Tr}\,\Pi_{f}).$ 
	Taking $t=0$ and $X_{0}=0,$ then $Z_{0}=\textup{exp}_{*_{0}}Y=Y$ and 
	formula \eqref{expression de la derivée de Pf....} follows.
\end{proof}
\begin{remarque}$\text{}$\label{remarque sur la derivée de rho, la finale j'espere...}
	 \begin{center}
		\begin{description}
			\item[$(i)$] Throughout the last proof, we have actually proved that the 
				``density map'' $\rho\,:\,\textup{Emb}(M,N){\rightarrow}\, 
				C^{\infty}(M,\mathbb{R}),\,f\mapsto f^{*}\mu^{f}/\mu$ 
				is a smooth tame map, and that its derivative at a point $f\in \textup{Emb}(M,N)$ in 
				direction $X\in\Gamma(f^{*}TN),$ is
				\begin{eqnarray}\label{formule derivée de rho}
					\rho_{*_{f}}X=\Big[\textup{div}_{f}(X^{\top}\circ f^{-1})-
					g(X^{\perp}\circ f^{-1},\textup{Tr}\,\Pi_{f})\Big]\circ f\cdot \rho(f).
				\end{eqnarray}
				Equation \eqref{formule derivée de rho} is a classical formula in 
				differential geometry (see for example \cite[p.158]{Jost} or \cite{Molitor-Lagrangian}).\\
			\item[$(ii)$] It may seem weird, in view of 
				\eqref{equation avec la trace de la seconde forme fondamentale blablabla} 
				where there are second order partial differentials, that $(P_{f})_{*_{X}}Y$ 
				is only a nonlinear partial differential operator of order 1 in $X.$ 
				This comes from the ``artificial'' splitting $Z=Z^{\top}+Z^{\bot}$ 
				(recall that $Z=\textup{exp}_{*_{X}}Y,$ see the proof above) 
				which introduces a first order nonlinear partial differential operator
				in $X,$ since projecting a tangent vector on the tangent space of $f_{X}(M)$ 
				consumes the first derivatives of $f_{X}$. 
	\end{description}\end{center}
\end{remarque}

	From Lemma \ref{expression locale de rho}, it follows that 
	$Q_{f}\,:\,\varphi_{f}(U_{f})\rightarrow \Gamma_{f}\big(TN\vert_{f(M)}\big)\oplus C^{\infty}(M,\mathbb{R})$ is a smooth tame map, 
	and one may try to invert it on a neighborhood of the zero section.
\begin{lemme}\label{lemme Q est un diffeo local}
	For $f\in\textup{Emb}_{\mu}(M,N)^{\times},$ the smooth tame map 
	$Q_{f}\,:\,\varphi_{f}(U_{f})\rightarrow \Gamma_{f}\big(TN\vert_{f(M)}\big)\oplus C^{\infty}(M,\mathbb{R})$ is invertible on an open 
	neighborhood of the zero section, and its local inverse is also a smooth tame map.
\end{lemme}
\begin{proof}
	The conditions required by the inverse function theorem of Nash-Moser 
	are that $(Q_{f})_{*_{X}}$ is invertible for all $X$ in a neighborhood of the zero 
	section, and also that the family of inverses forms a smooth tame map (see \cite[Thm 1.1.1]{Hamilton}).\\
	So, let us take $X\in \varphi_{f}(U_{f}),$ $Y\in\Gamma\big(f^{*}TN\big)\,,$ 
	$Z_{\mu}\in \Gamma_{\mu}\big(TN\vert_{f(M)}\big)$ and $p_{Z}\in C^{\infty}(M,\mathbb{R})\,.$\\
	Denoting $Y=(Y_{\mu}+\textup{grad}(p)+p\cdot\textup{Tr}\,\Pi_{f})\circ f,$ we have :
	\begin{eqnarray}
	&&{(Q_{f})}_{*_{X}}Y=(Z_{\mu},p_{Z})\label{equation de la derivée de P au debut de la ligne}\\\
	&\Leftrightarrow &(Y_{\mu},(P_{f})_{*_{X}}Y)=(Z_{\mu},p_{Z})\nonumber\\
	&\Leftrightarrow& Y_{\mu}=Z_{\mu}\,\,\,\,\,\textup{and}\,\,\,\,\,(P_{f})_{*_{X}}Y=p_{z}\nonumber\\
	&\Leftrightarrow& \left \lbrace
	\begin{array}{ccc}\label{equation du systeme avec la derivée de P}
	Y_{\mu}=Z_{\mu},&&\\
	(P_{f})_{*_{X}}(\textup{grad}(p)+p\cdot\textup{Tr}\,\Pi_{f})\circ f=p_{z}-(P_{f})_{*_{X}}(Z_{\mu}).&&
	\end{array}
	\right.
	\end{eqnarray}
	From the equivalence between equation \eqref{equation de la derivée de P au debut de la ligne} 
	and equation \eqref{equation du systeme avec la derivée de P}, we see that $(Q_{f})_{*_{X}}$ is invertible 
	if and only if the operator $(P_{f})_{*_{X}}(\textup{grad}(p)+p\cdot\textup{Tr}\,\Pi_{f})\circ f$ acting on $p$ is invertible.\\
	Now, according to Lemma \ref{expression locale de rho}, $(P_{f})_{*_{X}}(\textup{grad}(p)+p\cdot\textup{Tr}\,\Pi_{f})\circ f$ 
	is a family of linear partial differential operators of degree 2 in $p$ 
	with coefficients which are nonlinear partial differential operators of degree 1 in $X.$ Moreover, as 
	\begin{eqnarray}
		 (P_{f})_{*_{0}}(\textup{grad}(p)+p\cdot\textup{Tr}\,\Pi_{f})\circ 
		f=(\triangle p-\Vert\textup{Tr}\,\Pi_{f}\Vert^{2}\cdot p)\circ f
	\end{eqnarray}
	is an elliptic operator in $p$ with analytical index zero (see the proof of Proposition 
	\ref{decomposition de Hel-Hod}), it follows by the topological invariance of the analytical index 
	that $(P_{f})_{*_{X}}(\textup{grad}(p)+p\cdot\textup{Tr}\,\Pi_{f})\circ f$ is a also elliptic in $p$ with 
	analytical index zero for $X$ sufficiently small. This family is also a family of injective 
	operators by the maximum principle (see \cite[p.96]{Aubin}). The maximum principle can be applied here, 
	because $(P_{f})_{*_{0}}(\textup{grad}(1) +1\cdot\textup{Tr}\,\Pi_{f})\circ f=-\Vert\textup{Tr}\,\Pi_{f}\Vert^{2}\circ f$ 
	is strictly negative (recall that we assume 
	$\textup{Tr}\,\Pi_{f}\neq 0$ for all $x$ in $M$), and thus, the term of order zero 
	$(P_{f})_{*_{X}}(\textup{grad}(1) +1\cdot\textup{Tr}\,\Pi_{f})\circ f$ will remain strictly negative for 
	small $X.$ It follows that this family is actually a family of invertible elliptic operators, and one 
	can apply \cite[Thm 3.3.1]{Hamilton} to deduces that this family of inverses forms a smooth tame family of 
	linear maps. The same conclusion being obviously true for the family of inverses $((Q_{f})_{*_{X}})^{-1},$ the lemma follows.
\end{proof}
	The fact that $\textup{Emb}_{\mu}(M,N)^{\times}$ is a submanifold of 
	$\textup{Emb}(M,N)^{\times}$ is now a simple consequence of Lemma \ref{lemme Q est un diffeo local} as we already remarked. 


\section{Mechanics on the space of volume preserving embeddings} 
\subsection{Euler-Lagrange equations on $\textup{Emb}_{\mu}(M,N)^{\times}$}\label{sss pas inspireeee}
	Let $M$ be a compact, connected and oriented 
	submanifold of a Riemannian manifold $(N,g)\,.$ We denote by $\mu$ the Riemannian volume form on $M$ induced by the ambient metric 
	$g\,.$\\

	In this section, we consider Lagrangian mechanics on $\textup{Emb}_{\mu}(M,N)^{\times}$ for 
	Lagrangians of the following type : 
	\begin{eqnarray}\label{eee definition lagr}
		 \widetilde{L}(X_{f}):=\int_{M}\,L\circ X_{f}\cdot \mu=\int_{f(M)}\,(L\circ X_{f}\circ f^{-1})\cdot \mu^{f}\,,
	\end{eqnarray}
	where $L\,:\,TN\rightarrow \mathbb{R}$ is a Lagrangian density and where 
	$X_{f}\in T_{f}\textup{Emb}_{\mu}(M,N)^{\times}\cong\Gamma_{\mu}(f^{*}TN)\,.$

	Observe that the last equality in \eqref{eee definition lagr} comes from a change of variables together 
	with the formula $f^{*}\mu^{f}=\mu\,.$\\

	In order to formulate the Euler-Lagrange equations on $\textup{Emb}_{\mu}(M,N)^{\times}$ 
	associated to $\widetilde{L}$, we have to introduce some terminology. Recall that the metric $g$ 
	induces a connector $K\,:\,T(TN)\rightarrow TN$ (see \cite{Lang}, chapter 10, page 284), 
	and that for $v_{x}\in T_{x}N\,,$ there is an isomorphism 
	\begin{eqnarray}
		T_{v_{x}}TN\overset{\cong}{\longrightarrow}T_{x}N\oplus T_{x}N,\,\,\,\xi\mapsto (\pi_{*_{v_{x}}}\xi, K\xi)\,,
	\end{eqnarray}
	where $\pi\,:\,TN\rightarrow N$ is the canonical projection. Hence, for 
	$\xi\in T_{v_{x}}TN\,,$ we have the decomposition $\xi=\xi^{h}+\xi^{v}$ which is characterized by
	$K\xi^{h}=0$ and $\pi_{*_{v_{x}}}\xi^{v}=0\,.$ This decomposition defines a splitting of 
	the bundle $T(TN)$ into a direct sum $T(TN)=HN\oplus VN\,,$ where $HN$ is the horizontal vector 
	bundle and $VN$ the vertical vector bundle (see \cite{Lang}).
	
	With this notation, for $v_{x}\in T_{x}N\,,$ we have :
	\begin{eqnarray}
		L_{*_{v_{x}}}\big\vert_{(HN)_{v_{x}}}\in (HN)_{v_{x}}^{*}\cong T_{x}^{*}N\cong T_{x}N\,.
	\end{eqnarray}
	Thus, there exists $(\nabla^{h}L)_{v_{x}}\in T_{x}N$ such that 
	\begin{eqnarray}
		L_{*_{v_{x}}}\xi^{h}=g\big((\nabla^{h}L)_{v_{x}},\pi_{*_{v_{x}}}\xi^{h}\big)\,,
	\end{eqnarray} 
	for all $\xi^{h}\in (HN)_{v_{x}}\,.$ Similarly, there exists $(\nabla^{v}L)_{v_{x}}\in T_{x}N$ such that
	\begin{eqnarray}
	L_{*_{v_{x}}}\xi^{v}=g\big((\nabla^{v}L)_{v_{x}},K\xi^{v}\big)\,,
	\end{eqnarray}
	for all $\xi^{v}\in (VN)_{v_{x}}\,.$ In this way, we define two maps 
	$\nabla^{h}L\,:\,TN\rightarrow TN$ and $\nabla^{v}L\,:\,TN\rightarrow TN$ which are 
	smooth and fiber preserving. For practical calculations, 
	let $``\,\,\overset{\sharp}{\text{}}\,\,"\,:\,TN\rightarrow T^{*}N$ 
	be the canonical isomorphism induced by the metric $g$ 
	and $``\,\,\overset{\flat}{\text{}}\,\,"\,:\,T^{*}N\rightarrow TN$ its inverse. 
	For $v_{x}\in T_{x}N\,,$ it is not hard to see that :
	\begin{description}
		\item[$\bullet$]  $\big((\nabla^{v}L)_{v_{x}}\big)^{\sharp}=
			\{T_{x}N\ni u_{x}\mapsto \frac{d}{dt}\big\vert_{0}\,L(v_{x}+t\cdot u_{x})\in \mathbb{R}\}\,,$
		\item[$\bullet$]  $\big((\nabla^{h}L)_{v_{x}}\big)^{\sharp}=
			\{T_{x}N\ni u_{x}\mapsto \frac{d}{dt}\big\vert_{0}\,L(U(t)),\,
			\text{where}\,\,\alpha:=\pi\circ U\,\,\,\,\text{is a smooth}\\
			\text{curve in}\,\,N\,\,\text{such that}\,\,\alpha(0)=x\,\,
			\text{and}\,\,\dot{\alpha}(0)=u_{x}\,\,\text{and where}\,\,U\,\,\,\text{is a smooth}\\
			\text{parallel vector field along}\,\,\alpha\,\,\text{such that}\,\,U(0)=v_{x}\}\,.$
	\end{description}
\begin{example}\label{esdnjfdsnfkd}
	If $L:=\frac{1}{2}g(\,.\,,\,.\,)-V\circ\pi\,,$ where $V$ 
	is a function on $N\,,$ then $(\nabla^{v}L)_{v_{x}}=v_{x}$ and 
	$(\nabla^{h}L)_{v_{x}}=-(\textup{grad}(V))_{x}$ for all $v_{x}\in T_{x}N\,.$
\end{example}
\begin{remarque}\label{legendre}
			Using the Legendre transform $\mathbb{F}L\,:\,TN\rightarrow T^{*}N$ 
			(see for example \cite{Abraham-Marsden}), one observes that 
			$(\nabla^{v} L)_{v_{x}}=\big((\mathbb{F}L)(v_{x})\big)^{\flat}\,.$
\end{remarque}
	Finally, let us introduce, for a given $f\in \textup{Emb}(M,N)^{\times},$ the following operator:
	\begin{eqnarray}
		\mathbb{P}_{f}\,:\,
		 \left \lbrace
		\begin{array}{ccc}
			\Gamma(f^{*}TN)\rightarrow\Gamma_{\mu}(f^{*}TN),&&\\
			X\mapsto\mathbb{P}_{f(M)}(X\circ f^{-1})\circ f,&&
	\end{array}
	\right.
	\end{eqnarray}
	(see Remark \ref{remarque sur la decomposition, la vraie de H-H} $(iii)$ for the definition of $\mathbb{P}_{f(M)}).$
\begin{proposition}
	 \label{theoreme ultime??? lol}
	The Euler-Lagrange equations on $\textup{Emb}_{\mu}(M,N)^{\times}$ associated to a Lagrangian density $L\,:\,TN\rightarrow\mathbb{R}$ 
	are :
	\begin{eqnarray}\label{equations d'EULER-lagrange}
		\mathbb{P}_{f}\big[\nabla_{{\partial_{t}}f}\big(\nabla^{v}L\big)_{{\partial_{t}}f}-(\nabla^{h}L)_{{\partial_{t}}f}\big]=0,
	\end{eqnarray}
	where $f=f_{t}$ is a smooth curve in $\textup{Emb}_{\mu}(M,N)^{\times}\,,$ and where $\nabla$ is the Levi-Civita connection associated to 
	$g\,.$ 
\end{proposition}
\begin{proof}
	Let $f_{t}$ be a smooth curve in $\text{Emb}_{\mu}(M,N)^{\times}$ and 
	let $\tilde{f}_{s}$ be a proper variation in $\text{Emb}_{\mu}(M,N)^{\times}$ of 
	the curve $f_{t},$ i.e. a variation with fixed ends 
	(see \cite{Abraham-Marsden}). We have :
	\begin{eqnarray}
		&&\dfrac{d}{ds}\bigg\vert_{0}\,\int_{a}^{b}\,
			\widetilde{L}(\partial_{t}\tilde{f})\,dt=
			\dfrac{d}{ds}\bigg\vert_{0}\,\int_{a}^{b}\int_{M}\,(L\circ\partial_{t}\tilde{f})\cdot \mu\,dt\nonumber\\
		&=&\int_{a}^{b}\int_{M}\,\partial_{s}\vert_{0}\,(L\circ\partial_{t}\tilde{f})\cdot  
			\mu\,dt\label{equation premier terme lsfffffff}
			=\int_{a}^{b}\int_{M}\,(L_{*_{\partial_{t}f}}\partial_{s}\vert_{0}\partial_{t}\tilde{f})\cdot  \mu\,dt\nonumber\\
		&=&\int_{a}^{b}\int_{M}\,\big[g_{f}((\nabla^{h}L)_{\partial_{t}f},\partial_{s}\vert_{0}\tilde{f})
			+g_{f}((\nabla^{v}L)_{\partial_{t}f},{K\,\partial_{s}\vert_{0}\,\partial_{t}\tilde{f}})\big]\cdot  
			\mu\,dt.\;\;\;\;\;\;\text{}\label{equation terme avec le connecteur}
	\end{eqnarray}
	As $K\,\partial_{s}\vert_{0}\,\partial_{t}\tilde{f}=
	\nabla_{\partial_{t}f}\partial_{s}\vert_{0}\tilde{f},$ the second term in
	\eqref{equation terme avec le connecteur} can be rewritten 
	\begin{eqnarray}
		&&\int_{a}^{b}\int_{M}\,g_{f}((\nabla^{v}L)_{\partial_{t}f},
			\nabla_{\partial_{t}f}\,\partial_{s}\vert_{0}\tilde{f})\cdot  \mu\,dt\nonumber\\
		&&=\int_{a}^{b}\int_{M}\,\big[\partial_{t}\,g_{f}((\nabla^{v}L)_{\partial_{t}f},
			\partial_{s}\vert_{0}\tilde{f})-g_{f}(\nabla_{\partial_{t}f}(\nabla^{v}L)_{\partial_{t}f},
			\partial_{s}\vert_{0}\tilde{f})\big]\cdot  \mu\,dt\nonumber\;\;\;\;\;\;\;\;\;\;\text{}\\
		&&=-\int_{a}^{b}\int_{M}\,g_{f}(\nabla_{\partial_{t}f}(\nabla^{v}L)_{\partial_{t}f},\partial_{s}
			\vert_{0}\tilde{f})\cdot  \mu\,dt.\label{equation qui tue sgsg3}
	\end{eqnarray}
	The proposition follows from \eqref{equation terme avec le connecteur}, \eqref{equation qui tue sgsg3} and the fact that 
	\eqref{decomposition H-H} is an orthogonal decomposition. 
\end{proof}
\begin{remarque}
	$\text{}$
	\begin{center}
	\begin{description}
		 \item[$(i)$] If the submanifold $ M\subseteq N$ is a point, 
			then \eqref{equations d'EULER-lagrange} reduces to a coordinate-free 
			formulation of the classical Euler-Lagrange equations on $N\,:$
			\begin{eqnarray}\label{pascal sur msn}
				\nabla_{\dot{\alpha}(t)}\big(\nabla^{v}L\big)_{\dot{\alpha}(t)}-(\nabla^{h}L)_{\dot{\alpha}(t)}=0,
			\end{eqnarray}
			where $\alpha$ is a smooth curve in $N.$ Equation \eqref{pascal sur msn} 
			is a particular case of a more general 
			free-coordinate formulation of the Euler-Lagrange equations using connections (see \cite{Gamboa-Solomin}).\\
		\item[$(ii)$] According to the above remark, the Euler-Lagrange equations 
			$\eqref{equations d'EULER-lagrange}$ on $\textup{Emb}_{\mu}(M,N)^{\times}$ are 
			simply the ``pointwise'' classical Euler-Lagrange equations twisted by the ``Helmholtz-Hodge projector'' $\mathbb{P}_{f}.$
	\end{description}                 
	\end{center}
\end{remarque}
	
	An alternative description of the Euler-Lagrange equations on $\textup{Emb}_{\mu}(M,N)^{\times}\,,$ which is straightforward 
	and maybe more explicit than the one given in Proposition \ref{theoreme ultime??? lol}, is as follows.
\begin{proposition}\label{proposition yoyoyoyoy}
	The Euler-Lagrange equations on $\textup{Emb}_{\mu}(M,N)^{\times}$ associated to a Lagrangian density $L\,:\,TN\rightarrow\mathbb{R}$ 
	are :
		\begin{eqnarray}\label{eee euler lagrange explicites}
			\nabla_{{\partial_{t}}f}\big(\nabla^{v}L\big)_{{\partial_{t}}f}-(\nabla^{h}L)_{{\partial_{t}}f}&=&
			\textup{grad}(p\circ f^{-1})+
			(p\circ f^{-1})\cdot \textup{Tr}\,\Pi_{f}\nonumber\\
			\textup{div}_{f}(\partial_{t}f^{\top})&=&g\big(\partial_{t}f^{\perp},\textup{Tr}\,
			\Pi_{f}\big)\,,
		\end{eqnarray}
	where $f=f_{t}$ is a smooth curve in $\textup{Emb}_{\mu}(M,N)^{\times}\,,$ $p=p_{t}\,:\,M\rightarrow \mathbb{R}$ 
	is a time-dependant function, $\textup{grad}(p\circ f^{-1})$ is the Riemannian gradient of $p\circ f^{-1}$ taken with respect 
	to the induced metric on $f(M)$ and where $\nabla$ is the Levi-Civita connection associated to $g\,.$ 
\end{proposition}
\begin{remarque}
	The ``pressure term" $p$ in \eqref{eee euler lagrange explicites} is uniquely determined by the equation 
	\begin{eqnarray}
		\triangle(p\circ f^{-1})-(p\circ f^{-1})\cdot \|\textup{Tr}\,\Pi_{f}\|^{2}&=&
			\textup{div}_{f}\,\Big(\big[\nabla_{{\partial_{t}}f}\big(\nabla^{v}L\big)_{{\partial_{t}}f}
			-(\nabla^{h}L)_{{\partial_{t}}f}\big]^{\top}\Big)\nonumber\\
		&&\text{}\,\,\,\,\,\,-g\Big(\big[\nabla_{{\partial_{t}}f}\big(\nabla^{v}L\big)_{{\partial_{t}}f}
			-(\nabla^{h}L)_{{\partial_{t}}f}\big]^{\perp},\textup{Tr}\,\Pi_{f}\Big)\,,
	\end{eqnarray} 
	where $\triangle$ is the Laplacian operator on $f(M)$ for the induced metric (see the proof of 
	Proposition \ref{decomposition de Hel-Hod}).
\end{remarque}

\subsection{Application: generalization of the Euler equations}\label{eeeeffffgggfggrrrr}
	By ``Euler equations", we are referring to the equations of an incompressible fluid on an oriented Riemannian manifold $(M,g)$
	with Riemannian volume form $\mu\,:$
	\begin{eqnarray}\label{euler equation au macdo}\label{les petites equations}
			\partial_{t}X_{t}+\nabla_{X_{t}}X_{t}&=&\textup{grad}(p_{t})\nonumber\\
			\textup{div}_{\mu}(X_{t})&=&0\,,
		\end{eqnarray}
	where $X_{t}$ is a time dependent vector field describing the velocity of the fluid, $p_{t}\,:\,M\rightarrow\mathbb{R}$ 
	is the pressure of the fluid, $\nabla_{X_{t}}X_{t}$ is the Riemannian covariant derivative of $X_{t}$ in direction $X_{t}$ 
	and where $\textup{grad}(p_{t})$ is the Riemannian gradient of $p_{t}\,.$ 
	The condition $\textup{div}_{\mu}(X_{t})=0$ guaranties that the fluid is incompressible. 
		
	It is known, since Arnold's paper \cite{Arnold}, that the above equations can be interpreted as geodesic 
	equations on the Fr\'{e}chet Lie group $\textup{SDiff}_{\mu}(M)$ for the right invariant $L^{2}$-metric which is defined, at the identity
	diffeomorphism, as $\langle X,Y\rangle:=\int_{M}\,g(X,Y)\cdot\mu\,,$
	where $X,Y\in \mathfrak{X}_{\mu}(M):=\{Z\in \mathfrak{X}(M)\,\big\vert\,
	\textup{div}_{\mu}(Z)=0\}$. Observe that the latter space is identified with the Lie algebra of $\textup{SDiff}_{\mu}(M)\,.$
	
	Briefly, the fact that the Euler equations \eqref{euler equation au macdo} are equivalent to 
	the geodesic equations of $\textup{SDiff}_{\mu}(M)$
	comes from the right-invariance of the metric $\langle\,,\,\rangle$ together with the following general fact: 
	geodesic equations on a Lie group for a right (or left) invariant metric are equivalent to 
	an evolution equation on the Lie algebra called \textit{Euler equation}\footnote{Their are 
	several formulations of the Euler equation. One of them is as follows. If $G$ is a Lie group with Lie algebra $\mathfrak{g}\,,$ 
	endowed with a right or left invariant metric $\langle\,,\rangle\,,$ then its associated Euler equation is $\dot{\alpha}(t)=
	\pm\textup{ad}^{*}(\alpha(t)^{\flat})(\alpha(t))\,,$ 
	where $\textup{ad}^{*}$ is the coadjoint representation of $G\,,$ 
	$\alpha(t)$ is a smooth curve in $\mathfrak{g}^{*}$ and where $\alpha^{\flat}(t)$ is the unique 
	curve in $\mathfrak{g}$ which satisfies $\alpha(t)(\xi)=\langle \alpha^{\flat}(t),\xi\rangle$ for 
	all $\xi\in \mathfrak{g}\,.$ The sign in front of $\textup{ad}^{*}$ depends on convention and whether the metric 
	is right or left invariant. One can show that the Euler equation is equivalent to the geodesic equations for the metric 
	$\langle\,,\rangle\,,$ see \cite{Arnold-Khesin}.}, which, 
	in the particular case $\textup{SDiff}_{\mu}(M)$ yields the Euler equations of an incompressible 
	fluid; this is Arnold's remarkable observation (see \cite{Arnold,Arnold-Khesin,Ebin-Marsden}). 

	For us, the important point is that the Euler equations \eqref{euler equation au macdo} 
	are equivalent to the geodesic equations for the metric $\langle\,,\,\rangle\,,$ and that this metric 
	can be naturally generalized to $\textup{Emb}_{\mu}(M,N)^{\times}\,,$ as follows :
	\begin{eqnarray}\label{definition metric sur emb}
		\langle X_{f},Y_{f}\rangle=\int_{M}\,g(X_{f},Y_{f})\cdot \mu\,,
	\end{eqnarray}
	where $f\in \textup{Emb}_{\mu}(M,N)^{\times}$ and where $X_{f},Y_{f}\in T_{f}\textup{Emb}_{\mu}(M,N)^{\times}\cong 
	\Gamma_{\mu}\big(f^{*}TN\big)\,.$ 

	It is well known in the context of Riemannian geometry that geodesics are solutions of the Euler-Lagrange equations 
	associated to the energy, which in our case reads 
	\begin{eqnarray}
		T\textup{Emb}_{\mu}(M,N)^{\times}\rightarrow \mathbb{R}\,,\,\,\,\,\,\,
		X_{f}\mapsto \dfrac{1}{2}\int_{M}\,g(X_{f},X_{f})\cdot \mu\,.
	\end{eqnarray}
	According to Proposition \ref{proposition yoyoyoyoy}, and taking into account Example \ref{esdnjfdsnfkd}, we thus get
\begin{proposition}\label{je derpime pas de fric}
	The geodesic equations on $\textup{Emb}_{\mu}(M,N)^{\times}$ for the metric $\langle\,,\,\rangle$ 
	defined in \eqref{definition metric sur emb} are
		\begin{eqnarray}\label{dsf fsd fddg}
				\nabla_{\partial_{t}f}\partial_{t}f&=&\textup{grad}(p\circ f^{-1})+(p\circ f^{-1})
				\cdot \textup{Tr}\,\Pi_{f}\nonumber\\
				\textup{div}_{f}(\partial_{t}f^{\top})&=&g\big(\partial_{t}f^{\perp},\textup{Tr}\,
				\Pi_{f}\big)\,,
		\end{eqnarray}
	where $f=f_{t}$ is a smooth curve in $\textup{Emb}_{\mu}(M,N)^{\times}\,,$ and where $p=p_{t}\,:\,M\rightarrow \mathbb{R}$ 
	is a time-dependant function. 
\end{proposition}
	In the special case $M=N\,,$ then $\textup{Emb}_{\mu}(M,N)=\textup{SDiff}_{\mu}(M)$ and the term 
	$\textup{Tr}\,\Pi_{f_{t}}$ in \eqref{dsf fsd fddg} vanishes.
	Moreover, If $f=f_{t}$ is a smooth curve in $\textup{Diff}(M)$ and 
	if $X_{t}:=(\partial_{t}f_{t})\circ f_{t}^{-1}\,,$ then one has the formula
	\begin{eqnarray}
		\nabla_{\partial_{t}f}\partial_{t}f=
		(\partial_{t}X_{t}+\nabla_{X_{t}}X_{t})\circ f\,,
	\end{eqnarray} 
	from which, together with $\textup{Tr}\,\Pi_{f_{t}}\equiv 0\,,$ one easily sees that \eqref{dsf fsd fddg} reduces to 
	the usual Euler equations of an incompressible fluid \eqref{les petites equations}. 
	Hence, \eqref{dsf fsd fddg} is indeed a generalization of the Euler equations; we call it 
	the \textit{Euler equations of an incompressible membrane}.

	\text{}\\\\
	\textbf{Acknowledgments.} I would like to thank Tudor Ratiu who brought to my attention 
	the space of volume preserving embeddings during a postdoctoral stay. 
	
	This work was partially supported by the Japan Society for the Promotion of Science.


\end{document}